\documentclass[12pt]{article}

\usepackage{latexsym}
\usepackage{amsmath}
\usepackage{amsfonts}
\usepackage{amssymb}
\usepackage{amsthm}
\usepackage{psfrag}
\usepackage{graphicx}
\usepackage{natbib}

\newcommand{\betabold}    {\mbox{\boldmath${\beta}$}}
\newcommand{\gammabold}    {\mbox{\boldmath${\gamma}$}}

\newcommand{\xibold}     {\mbox{\boldmath${\xi}$}}
\newcommand{\etabold}     {\mbox{\boldmath${\eta}$}}

\newcommand{\real}{\ensuremath{\mathbb{R}}}

\newtheorem{lem}{Lemma}
\newtheorem{cor}{Corollary}

\title{On the Identifiability of the Functional Convolution Model}
\author{Giles Hooker}
\date{}

\begin{document}

\maketitle

\begin{abstract}
This report details conditions under which the Functional Convolution Model described in \citet{AHG13} can be identified from Ordinary Least Squares estimates without either dimension reduction or smoothing penalties. We demonstrate that if the covariate functions are not spanned by the space of solutions to linear differential equations, the functional coefficients in the model are uniquely determined in the Sobolev space of functions with absolutely continuous second derivatives.
\end{abstract}

\citet{AHG13} introduced the {\em Functional Convolution Model} (FCM) in which for each observation $i = 1,\ldots,n$ a functional response $Y_i(t)$ depends on the short-term history of one or more functional covariates $X_{ij}(t)$, $j = 1,\ldots,p$ which are measured on the same time domain $t \in [0, T_i]$ along with scalars $z_{ik}$ for $k = 1,\ldots,d$. This is expressed mathematically as
\begin{equation} \label{convolution}
y_i(t) = \beta_{00} + \sum_{j = 1}^{d} \beta_{0k} z_{ik} + \sum_{j=1}^p  \int_{0}^{\alpha_j} \beta_j(u)
x_{ij}(t-u)du + \epsilon_i(t), \ i = 1,\ldots,n.
\end{equation}
Here $y_i(t)$ responds  to the past $\alpha_j$ time units of
$x_{ij}$ via a functional linear model \citep{RamsaySilverman05} and this relationship
is constant over $t$. The $\epsilon_i(t)$ are assumed to be mean-zero stochastic processes with stationary covariance. Note that
while the $y_i(t)$ and the $x_{ij}(t)$ must share the same time domain, this
domain needs not be the same across different observations. We have
parameterized the model so that $\beta_j(0)$ represents the instantaneous
effect of $x_{ij}(t)$ on the response at time $t$. $\betabold_0 = (\beta_{00},\beta_{01},\ldots,\beta_{0d})$ represent coefficients for scalar covariates including the intercept.

An estimate of the parameters in the FCM via Ordinary Least Squares (OLS) was proposed, obtaining estimates for $\betabold(u) = (\betabold_0,\beta_1(u),\ldots,\beta_p(u))$ that minimize
\[
\mbox{SSE}(\betabold) = \sum_{i=1}^n \int_{\alpha^*}^{T_i} \left( y_i(t) -
\sum_{j = 1}^{d} \beta_{0j} z_{ij}  - \sum_{j=1}^p  \int_{t-\alpha_j}^t \beta_j(u) x_{ij}(t-u)du
\right)^2 dt
\]
for $\alpha^* = \max(\alpha_1,\ldots,\alpha_p)$. This has been chosen so that the range of $t$ in the outer integral ensures that
the range of $t-u$ within the squared term does not go below $0$. To this criterion, \citet{AHG13} added smoothing penalties for each of the $\beta_j(u)$ and represented them via a basis expansion.

The identifiability of the FCM under the OLS without penalization, and with arbitrarily complex basis expansions, is not clear.  The model can be placed between the scalar response model $y_i = \beta_0 + \int \beta_1(t) x_i(t) dt + \epsilon_i$ in which $\beta_1(t)$ cannot be identified without smoothing or dimension reduction and the concurrently linear model for functional responses $y_i(t) = \beta_0(t) + \beta_1(t) x_i(t) + \epsilon_i(t)$ in which $\beta_0(t)$ and $\beta_1(t)$ can be obtained from a linear regression at each time $t$ (although smoothing can still serve to reduce variance). The purpose of this report is to investigate under what circumstances both smoothing penalty and basis expansion can be removed. That is, restricting the $\beta_j(u)$ to lie in a Sobolev class of functions, under what conditions  does $\mbox{SSE}(\betabold)$ have a unique minimum without further restructions? In general, the design of covariate functions  in functional data analysis
has received little attention. The functional convolution models studied
here present a challenge in that identifiability depends on the finite-sample
design of the covariates. 

For this report, we assume our covariate processes $\beta_j(t)$ lie in the Sobolev space $W[0, \ \alpha_j]$ of functions defined
on $[0, \ \alpha_j]$ for which all second derivatives are absolutely continuous.
For convenience, we also assume that the $y_i(t)$ and $x_i(t)$ have been centered by
their integral averaged over all observations and ignore $\beta_0$. We also assume that the $Y_i$ have been centered by their time-series mean and no other scalar covariates are present in order to remove the intercept from the model.
Within this space, minimizing $\mbox{SSE}(\betabold)$ is equivalent to setting its Gateaux derivative to zero; that is, solving the variational problem
\begin{equation} \label{variational}
<\gammabold,G[\betabold]> = F(\gammabold), \ \forall \gammabold = (\gamma_1,\ldots,\gamma_p)
\in W[0, \ \alpha_1] \otimes \cdots \otimes W[0, \ \alpha_p]
\end{equation}
where
\[
F(\gamma) = \sum_{j=1}^p \sum_{i=1}^n \int_0^{\alpha_j} \gamma_j(u)
\int_{\alpha^*}^{T_i} x_{ij}(t-u) y_i(t) dt du,
\]
and
\[
G[\betabold] = \sum_{j,k=1}^p \sum_{i=1}^n \int_0^{\alpha_j} \int_{\alpha^*}^{T_i}
x_{ik}(t-v) x_{ij}(t-u) dt \beta_j(u) du.
\]
where the inner product is taken as the product $L^2$ inner product on
square-integrable functions
\[
<\gammabold,\betabold> = \sum_{j=1}^p \int_0^{\alpha_j} \gamma_j(u) \beta_j(u) du.
\]
The identification of the OLS estimates is now equivalent to the invertibility of $G$. In particular, $\betabold$ will
be uniquely identified if
\[
<\gammabold,G[\betabold]> = 0, \  \forall \gammabold \Rightarrow \betabold = 0
\]
and in particular if
\begin{equation} \label{identifiability}
<\betabold,G[\betabold]> =  \int_{\alpha^*}^{T_j} \sum_{i=1}^n  \left[\int_0^{\alpha_j}
x_{ij}(t-v) \beta_j(v) dv \right] \left[ \int_0^{\alpha_j} x_{ij}(t-u) \beta_j(u) du \right] dt > 0
\end{equation}
for all $j$ and all non-zero $\beta_j$.

In particular, we can
let $\xibold_1,\xibold_2,\ldots$ form a basis for $ W[0, \ \alpha_1] \otimes \cdots \otimes W[0, \ \alpha_p]$ with $\xibold_j = (\xi_{j1},\ldots,\xi_{jp})$ and \eqref{identifiability} reduces to the requirement that for every $j$ and $l$,
\begin{equation} \label{identifiability2}
\int_{0}^{\alpha_j} \xi_{jl} (u) x_{ij}(t-u) du \neq 0
\end{equation}
for $t$ in a subset of $[\alpha^*, T_i]$ of positive measure for at least one $i$.  This condition is not readily checked, particularly in real-world applications since it requires checking an infinite collection of inner-products; however the non-identifiability of a design can be readily assessed. Because of this \citet{AHG13} employed smoothing parameters to ensure the identifiability of their estimates. However, it is possible to characterize designs such that the collection
\[
x_{ijt}(u) = x_{ij}(t-u)
% : [0 \ \alpha_j] \rightarrow \R
\]
spans a finite-dimensional space as $t$ is varied; i.e. for which there is a finite collection of functions
$\eta_{1}(u),\ldots,\eta_K(u)$ such that
\[
x_{ijt}(u) = \sum_{k=1}^K c_k(t) \eta_k(u)
\]
for all $t$. This set of self-similar functions can be expressed as solutions to a linear differential equation. In the lemma below we restrict to a single real-valued function for the sake of clarity and set $\alpha_j = 1$

\begin{lem} \label{findimlem}
Let $x(t)$ have continuous first derivatives, then $x$ satisfies
\begin{equation} \label{selfsimilar}
x(t+u) = \sum_{k=1}^K \zeta_k(t) \eta_k(u)
\end{equation}
for all $t$ and $\eta_k: [0, \ 1] \rightarrow \real$,  if and only if
\begin{equation} \label{selfsimfns}
x(t) = \sum_{k=1}^K c_k  t^{m_k} e^{a_k t} \sin( b_k t + d_k)
\end{equation}
for real-valued constants $(a_k,b_k,c_k,d_k)$, and integers $m_k$, $k = 1,\ldots,K$.
\end{lem}
\begin{proof}
We observe that
\[
x(t+u+dt) = \sum_{k=1}^K \zeta_k(t+dt) \eta_k(u) =  \sum_{k=1}^K \zeta_k(t) \eta_k(u+dt)
\]
and thus
\begin{equation} \label{selfsimderiv}
x'(t+u) = \sum_{k=1}^K \zeta_k'(t) \eta_k(u)=  \sum_{k=1}^K \zeta_k(t) \eta_k'(u)
\end{equation}
where $x'(\cdot)$ is the derivative taken with respect to its argument.  We can now examine which $\zeta_k$ satisfy the second equality in  \eqref{selfsimderiv} by restricting to a finite set of values for $u$.  Defining a set of evaluation points $u_l = (l-1)/K$ for $k = 1,\ldots,K$ we can produce matrices
\[
X_{lk} = \eta_k \left( u_l \right), \  \dot{X} = \eta'_k \left( u_l \right).
\]
Note that these do not depend on $t$. The last equality in \eqref{selfsimderiv} restricted to $u_1,\ldots,u_K$ now defines the differential equation
\begin{equation} \label{linode}
\frac{d}{dt} \etabold = \left[ X^{-} \dot{X} \right] \etabold
\end{equation}
where $X^{-}$ is a generalized inverse. Solutions to \eqref{linode} have
general form of \eqref{selfsimfns} \citep[e.g.,][]{BorelliColeman}, and thus
$x(t)$ must at least be of this form. It is easy to check that any function of
the form \eqref{selfsimfns} satisfies \eqref{selfsimilar} for some $K$ and
$(\zeta_k,\eta_k, k = 1,\ldots,K)$, completing the converse implication.
\end{proof}

From this we directly obtain the following

\begin{cor}
The variational problem \eqref{variational} has a unique solution in $W[0, \ \alpha_1] \otimes \cdots \otimes W[0, \ \alpha_p]$ if and only if
\[
\sum_{i=1}^n \int_{\alpha^*}^{T_i} \left( x_{ij}(t) - \sum_{k=1}^{K_{ij}} c_{ijk} t^{m_{ijk}} e^{a_{ijk}t} \sin(b_{ijk}t + d_{ijk}) \right)^2 dt >0
\]
for each $j$ and any finite choice of $K_{ij}$, $a_{ijk}$, $b_{ijk}$, $c_{ijk}$, $d_{ijk}$ and $m_{ijk}$.
\end{cor}

It is not difficult to provide example designs that satisfy
\eqref{identifiability2}. Consider the basis of periodic functions on $[0, \ 1]$
given by $\xi_k(u) = \sin(2k \pi t)$ then taking $x(t)$ to be defined on $[0, \ 1]$,
say, with
\[
x(t) = \sum 2^{-k} \xi_k(t)
\]
has non-zero inner product on $[0, \ 1]$ with $\xi_k(u)$ for each $k$. We note
that the range of $x(t)$ need not be restricted to $[0, \ 1]$. However, between
the finite-dimensional design described in Lemma \ref{findimlem} and the
identifiable design, it is possible to find $x_i(t)$ that is orthogonal to an
infinite dimensional subspace. Continuing our example, setting
\[
\tilde{x} = \sum 2^{-4k} \xi_{2k}(t)
\]
will yield $<\tilde{x},\xi_k>=0$ for all $k$ odd provided the domain of $x$ is
of integer length. Evaluating identifiability based on finite-dimensional
approximations as given, for example in the appendix of \citet{AHG13} is
straightforward. However, it seems more challenging to provide a protocol for
designing experiments for which this model will be employed.

Beyond conditions for identifiability, the implication of these calculations for the convergence rates of the FCM remain un-investigated. As an alternative to the OLS formulation, we can down-sample the data in the FCM to consider only the values of $y_i(t)$ at $U$ intervals. That is we observe that $y_i(lU)$ for $l = 1,\ldots,n_i=T_i/U$ taken at the discrete times $lU$ exactly follow a functional linear model. Hence as either $T_i \rightarrow \infty$ or $n \rightarrow \infty$, results in \citet{ZhangYu08} provide the consistency of our estimates when combined with smoothing penalties. Although such penalties needs not be required for the FCM, they were employed in \citet{AHG13} and we speculate that in fact the convergence rates that can be obtained for the FCM are no better than the functional linear model. 

\bibliographystyle{chicago}
\bibliography{grant}

\end{document}